\pdfoutput=1
\documentclass{article}
	\usepackage[14pt]{extsizes}
	\usepackage[utf8]{inputenc}
	\usepackage{amsfonts}
	\usepackage{parskip}
	\usepackage{cmap}
	\usepackage{cite,enumerate,float,indentfirst}
	\usepackage[russian]{babel}
	\usepackage{setspace,amsmath}
	\usepackage{amsthm,amssymb}
	\usepackage{hyperref}
	\usepackage[left=20mm, top=15mm, right=15mm, bottom=15mm, nohead, footskip=10mm]{geometry}
	\renewcommand{\Im}{\mathrm{Im}\,}
	\renewcommand{\Re}{\mathrm{Re}\,}
	
	\newcommand{\conj}[1]{\overline{#1}}
	\theoremstyle{plain}
	\newtheorem{proposition}{Утверждение}
	\newtheorem{definition}{Определение}
	\newtheorem{theorem}{Теорема}
	\newtheorem{lemma}{Лемма}
	
\begin{document}
		\begin{titlepage}
			\vspace{0.5cm}
			\begin{center}
				\Large{\textbf{Однородные CR-многообразия в $\mathbb{C}^4$}}
			\end{center}
				\vspace{1.3cm}
			\begin{abstract}
				In this paper we study holomorphically homogeneous model submanifolds CR-type (1, 3) complex space $\mathbb C^4$. One finds moduli space of five-dimensional model surfaces Bloom-Graham type ((2, 1), (3, 1), (4, 1)). It is shown that there exists unique model surface of this type with property of holomorphical homogeneous, which is equivalent to tube surface $\mathcal C$ with affin homogeneous base. One describes and classifies with respect to model surfaces the orbits relative to the group of holomorhical automorphisms of $\mathcal C$
				
				-----
				
				В работе исследуются однородные модельные подмногообразия CR-типа (1, 3) комплексного пространства $\mathbb C^4$. Найдено пространство модулей пятимерных модельных поверхностей типа ((2, 1), (3, 1), (4, 1)) по Блуму~"--- Грэму. Показано, что среди модельных поверхностей этого типа свойством голоморфной однородности обладает ровно одна поверхность, которая эквивалентна трубчатой поверхности $\mathcal {C}$ над аффинно-однородной кривой в $\mathbb R ^4$. Далее описано семейство шестимерных голоморфно-однородных поверхностей, получающихся как орбиты действия группы голоморфных автоморфизмов $\mathcal C$ и проведена их классификация с точки зрения соответствующих модельных поверхностей.
			\end{abstract}
			\vspace{1cm}

		\end{titlepage}
	\section{Введение}
	\qquad Интерес к голоморфно-однородным поверхностям восходит к Картану, который классифицировал локально однородные вещественные аналитические трехмерные поверхности в $\mathbb{C}^2$ \cite{Cartan}. Позднее классификация однородных поверхностей проводилась для четырехмерных поверхностей \cite{BeloshapkaKossovskiyClassification} и пятимерных гиперповерхностей \cite{LobodaHomogeneous}. В этой работе будет показано, что среди порождающих пятимерных однородных модельных поверхностей в $\mathbb{C}^4$ конечного типа по Блуму~"--- Грэму возможны только вполне невырожденные и $$\mathcal C = \begin{cases}
	x_2 = x_1^2,\\
	x_3 = x_1^3,\\
	x_4 = x_1^4,\\
	x_1, y_1, y_2, y_3, y_4 \in \mathbb{R},
\end{cases}$$ 
которая является трубчатой поверхностью над аффинно-однородной кривой в $\mathbb R^4$.
\smallskip

\qquad Напомним определение типа по Блуму~"--- Грэму. Рассмотрим росток $M_\xi$ порождающего CR-подмногообразия $M$ в точке $\xi\in\mathbb{C}^{n+k}$. В точке $\xi$ имеется комплексное касательное пространство $T^c_\xi M_\xi := T_\xi M_\xi \cap i T_\xi M_\xi \subset T_\xi M_\xi.$ Пусть 
\[\dim_\mathbb C T^c_\xi M_\xi = n,\;\mathrm{codim}\,_\mathbb R T_\xi M_\xi = k,\]
то есть $M_\xi$ имеет CR-тип $(n, k)$. Определим последовательность распределений: 
\[D_1 = T^c M,\]
\[D_{\nu+1} = D_\nu + [D_{\nu}, D_1],\; \nu = 1, 2, \ldots.\]
Для любого $\nu,$ распределение $D_\nu$ в точке $\xi$ задает подпространство $D_\nu (\xi)\subset T_\xi M$. Более того, набор $D_\nu(\xi)$ образует неубывающую систему: $${{D_{\nu}(\xi) \subset D_{\nu+1}(\xi)}},$$ следовательно, последовательность стабилизируется. Определим последовательность пар $(m_\nu, k_\nu),$ где $m_\nu$ -- $\nu$-ый номер пространства $D_{m_\nu}(\xi)$, для которого произошел скачок размерности, а $$k_\nu = \dim_\mathbb{R} D_{m_\nu}(\xi) - \dim_\mathbb{R} D_{m_\nu -1}(\xi)$$ -- величина скачка. 
\begin{definition}
	Типом по Блуму~"--- Грэму ростка $M_\xi$ подмногообразия в точке $\xi$ назовем $$m = ((m_1, k_1), (m_2, k_2),\ldots,(m_\nu, k_\nu), \ldots).$$ Тип в точке $\xi$ бесконечный по Блуму~"--- Грэму, если последовательность $D_\nu (\xi)$ стабилизируется на собственном касательном подпространстве, иначе~"--- конечным.
\end{definition}
\qquad Таким образом, для конечного типа: $\sum_\nu k_\nu = k$. Заметим, что при этом существует бесконечно много различных конечных типов по Блуму~"--- Грэму, однако в дальнейшем нас будет интересовать среди них те, чьи модельные поверхности голоморфно-однородны.
\begin{definition}
	Модельная поверхность $Q$ -- невырождена в точке $\xi\in Q$, если она голоморфно невырождена и имеет в точке $\xi$ конечный тип по Блуму~"--- Грэму.
\end{definition}
\begin{definition}
	CR-многообразие $M$ -- невырождено в точке $\xi\in M$, если его модельная поверхность в точке $\xi$ -- невырождена в соответствующей точке.
\end{definition}
\qquad Для невырожденного голоморфно-однородного многообразия тип по Блуму~"--- Грэму всюду постоянен. Кроме того, в \cite{FiniteBlGr} показано, что для однородных модельных поверхностей на каждом шаге происходит рост размерности $D_{\nu}(\xi)$. Поэтому для подмногообразий в $\mathbb{C}^4$ коразмерности $3$ с голоморфно-однородной модельной поверхностью возможно только 2 типа по Блуму~"--- Грэму:
\begin{itemize}
	\item ((2, 1), (3, 2))
	\item ((2, 1), (3, 1), (4, 1)).
\end{itemize}
\smallskip
Первый тип соответствует ростку вполне невырожденной поверхности. В некоторых координатах его модельная поверхность имеет вид:
\[\begin{cases}
	\Im w_2 = |z|^2,\\
	\Im w_3 = 2\Re z^2\conj z,\\
	\Im w_4 = 2\Im z^2\conj z,
\end{cases}\]
и известно, что она является голоморфно-однородной~\cite{UniversalModel}. Отметим, что в этом случае модельная поверхность единственна. В то же время, как будет показано ниже, для другого типа это не так. В дальнейшем будем рассматривать второй тип.

\section{Нормальная форма ростка и пространство модулей модельных поверхностей}

Рассмотрим росток $M_\xi$ типа $((2, 1), (3, 1), (4, 1))$ по Блуму~"--- Грэму. Введем градуировку:
\[ [z] = [\conj z] = 1,\: [u_2] = 2,\: [u_3] = 3,\: [u_4] = 4.\]
Согласно \cite{BloomGraham}, общий вид $M_\xi$ в некоторых координатах задается системой:
\[
\begin{cases}
	\Im w_2 = a z \conj{z} + \ldots,\\
	\Im w_3 = 2 \Re \alpha z^2 \conj{z} + \ldots,\\
	\Im w_4 = b |z|^4 + 2 \Re \beta z^3 \conj{z} + 2 \Re \gamma z^4 + 2 \Re \delta z^2 u_2 +\\ \qquad\qquad\qquad\qquad\qquad+ c |z|^2 u_2 + 2 \Re \chi z u_3 + d u_2^2 + \ldots ,
\end{cases}
\]
где $a, \alpha \ne 0, \, (b, c, d, \beta,\gamma, \delta, \chi) \ne 0,\, (a, b, c, d)\in\mathbb R^4,\, (\alpha, \beta,\gamma,\delta,\chi)\in\mathbb C^5$, а многоточие означает члены большего веса относительно введенной градуировки. При $\alpha = r^3 \exp(i \varphi)$ заменой координат: $z \to z\exp(-i \varphi)/r,\, w_2 \to w_2 a/r^2,\, w_3 \to w_3,\, {w_4 \to w_4}$, можно привести систему к виду:
\[
\begin{cases}
	\Im w_2 = z \conj{z} + \ldots,\\
	\Im w_3 = 2 \Re (z^2 \conj{z}) + \ldots,\\
	\Im w_4 = \tilde b |z|^4 + 2 \Re (\tilde\beta z^3 \conj{z}) + 2 \Re (\tilde\gamma z^4) +\\ \qquad\qquad+ 2 \Re (\tilde\delta z^2 u_2) + \tilde c |z|^2 u_2 + 2 \Re (\tilde\chi z u_3) + \tilde d u_2^2 + \ldots
\end{cases}
\]

Далее делаем замену: $z \to z,\, w_2 \to w_2,\, w_3 \to w_3,\, w_4 \to w_4 + \tilde\gamma z^4 + \tilde \delta z^2 w_2 + \tilde \chi z w_3 +( \tilde d/2 - i\tilde c/4 ) w_2^2 $.
Итого, вводя новые обозначения для коэффициентов при $z^4$ и $z^3\conj z$, мы получаем, что произвольный росток типа $((2, 1), (3, 1), (4, 1))$ в некоторой системе координат имеет вид:
\[
\begin{cases}
	\Im w_1 = z\conj z + \ldots,\\
	\Im w_2 = 2 \Re z^2\conj z + \ldots,\\
	\Im w_3 = a |z|^4 + 2 b\, \Re z^3 \conj z  - 2 c \Im z^3 \conj z + \ldots,
\end{cases}
\]
Причем $(a, b, c) \ne 0$.

Далее мы будем рассматривать поверхность $\mathcal{Q}_{a, b, c}$:
\[
\begin{cases}
	\Im w_2 = z\conj z,\\
	\Im w_3 = 2 \Re z^2\conj z,\\
	\Im w_4 = a |z|^4 + 2b \,\Re z^3 \conj z - 2c \,\Im z^3\conj z,
\end{cases}
\]
\quad Поверхность $\mathcal Q_{a,b,c}$, как и всякая модельная поверхность~\cite{FiniteBlGr}, обладает следующим свойством. Рассмотрим $\mathrm{aut}\, M_\xi$ алгебру Ли ростков вещественных векторных полей с голоморфными коэффициентами в точке $\xi$, касательных к ростку $M_\xi$. В координатах $(z,w) = (z, w_2, w_3, w_4)\in \mathbb C^4$:
\[\mathrm{aut}\, M_\xi = \left\{X = 2\Re\left(f(z, w)\frac{\partial}{\partial z} + g_2(z, w)\frac{\partial}{\partial w_2} + g_3(z, w)\frac{\partial}{\partial w_3} + g_4(z, w)\frac{\partial}{\partial w_4}\right)\right\},\]
где $f, g_2, g_3, g_4$ -- ростки функций, голоморфных в точке $\xi$, $X = X(z, w)$ -- росток касательного поля к $M_\xi$.  Тогда для любой $Q = Q_{a,b,c}$ размерность $\mathrm{aut}\, Q_\xi$ мажорирует размерность $\mathrm{aut}\, M_\xi$, если $M_\xi$ имеет тип по Блуму~"--- Грэму $((2, 1), (3, 1), (4, 1))$ в точке $\xi$, то есть справедливо равенство:
\[\dim \mathrm{aut}\, M_\xi \leqslant \dim\mathrm{aut}\, Q_\xi.\] 	 
Займемся вопросом о попарной эквивалентности $Q_{a, b, c}$.
\begin{proposition}
	Биголоморфизм между двумя поверхностями $\mathcal{Q}_{a, b, c}$ и $\mathcal{Q}_{a'\!,\, b'\!,\, c'}$  существует тогда и только тогда, когда  $(a: b: c) = (a': b': c')$. Иными словами, пространством модулей модельных поверхностей типа $((2, 1), (3, 1), (4, 1))$ по Блуму~"--- Грэму является $\mathbb{R}P^2$.
\end{proposition}
\begin{proof}
	Очевидно, что если мы умножим тройку $(a, b, c)$ на ненулевое вещественное число, то соответствующая поверхность будет биголоморфно эквивалентна исходной. Введем на пространстве параметров следующее отношение эквивалентности:
	\[(a, b, c)\sim (a'\!,\, b'\!,\, c') \Leftrightarrow \exists \lambda \in \mathbb{R}\setminus \{0\} : \lambda(a, b, c) = (a'\!,\, b'\!,\, c').\]
	
	\quad Такое отношение не склеивает различные (с точностью до биголоморфизма) поверхности. Покажем, что различным точкам фактормножества параметров (то есть точкам из $\mathbb R P^2$) соответствуют неэквивалентные поверхности.

	\quad Пусть $(a  : b  : c) \ne (a'  : b': c')$, обозначим первую поверхность буквой $\mathcal{Q}$, а вторую - $\mathcal{Q}'$. Тогда $\mathcal Q$ биголоморфно эквивалентна $\mathcal Q '$ тогда и только тогда, когда они эквивалентны посредством отображения :
	\[
	\begin{cases}
		z \to \beta z,\\
		w_2 \to \beta_2 w_2,\\
		w_3 \to \beta_3 w_3,\\
		w_4 \to \beta_4 w_4 + \gamma_4 w_2^2,
	\end{cases}
	\]
	согласно \cite{FiniteBlGr}. Выпишем условие эквивалентности:

	\[
	\begin{cases}
		\Im \beta_2 w_2 = |\beta z|^2,\\
		\Im \beta_3 w_3 = 2\Re (\beta z)^2\conj{\beta z},\\
		\Im (\beta_4 w_4 + \gamma_4 w_2^2) = a'|\beta z|^4 + 2b'\Re(\beta z)^3\conj{\beta z} + 2c'\Im(\beta z)^3\conj{\beta z},
	\end{cases}
	\]
	при $w_2 = u_2 + i|z|^2, w_3 = u_3 + 2 i \Re z^2\conj z, w_4 = u_4 + i(a|z|^4 + 2b\, \Re z^3\conj z - 2c \, \Im z^3\conj z) $.
	Из первых двух уравнений получаем: $\Im\beta = \Im \beta_2 = \Im\beta_3 = 0,\, \Re\beta_2 = |\beta|^2,\, \Re\beta_3 = |\beta|^3$. Из последнего уравнения видно, что $\Im\beta_4 = \gamma_4 = 0$ и
	\[
	\begin{cases}
		a\beta_4 = a'\beta^4,\\
		b\beta_4 = b'\beta^4,\\
		c\beta_4 = c'\beta^4.
	\end{cases}
	\]
	Последняя система в случае $\beta_4 \ne 0$, равносильна $(a, b, c) = \beta^4/\beta_4(a'\!,\, b'\!,\, c')$, что противоречит условию $(a : b : c) \ne (a': b':c')$.
\end{proof}	
	\section{Голоморфно-однородные модельные поверхности}

Теперь найдем среди $\mathcal{Q}_{a,b,c}$ все голоморфно-однородные поверхности.
\begin{proposition}
	Любая голоморфно-однородная модельная поверхность $\mathcal{Q}_{a,b,c}$ биголоморфна $\mathcal{C}$.
\end{proposition}
\begin{proof}
	Голоморфная однородность означает, что для любой точки поверхности $\mathcal{Q}_{a,b,c}$ существует голоморфный автоморфизм, переводящий начало координат в эту точку. В \cite{FiniteBlGr} показано, что необходимым и достаточным условием голоморфной однородности модельной поверхности является постоянство Блум~"--- Грэм типа. Посмотрим на Блум~"--- Грэм тип в окрестности точки $(z^0, w^0) \in \mathcal{Q}_{a,b,c}$:
	\[\begin{cases}
		\Im w_2 + \Im w_2^0 = |z|^2 + 2\Re z \conj {z^0} + |z^0|^2,\\
		\Im w_3 + \Im w_3^0 = 2\Re z^2 \conj z + 2|z|^2\Re z^0 + 2\Re z^2\conj {z^0} +\\\qquad\qquad\qquad+ 2\Re z(\conj {z^0})^2 + 2|z^0|^2\Re z + 2\Re (z^0)^2\conj {z^0},\\
		\Im w_4 + \Im w_4^0 = a|z+z^0|^4 + 2b \,\Re((z+z^0)^3\conj{(z+z^0)})-2c \,\Im((z+z^0)^3\conj{(z+z^0)}).
	\end{cases} \]
	Сделаем замену:
	\[
	\begin{cases}
		\tilde w_2 = w_2 - 2i z\conj {z^0},\\
		\tilde w_3 = w_3 - 2i\tilde w_2 \Re z^0 - 2i z^2 \conj {z^0} - 2i z \conj {(z^0)}^2 - 2 i z |z^0|^2,\\
	\end{cases}
	\]
	и тогда первые два уравнения примут вид:
	\[
	\begin{cases}
		\Im \tilde w_2 = |z|^2,\\
		\Im \tilde w_3 = 2\Re z^2\conj z.\\
	\end{cases}
	\]
	А вот в третьем уравнении возникает слагаемое $\Re (z^2\conj z (2a\conj {z^0} + 3 b z^0)) - 3c \,\Im (z^2\conj z z^0)$. Оно не является гармоническим для произвольных $a, b, c$. Также нетрудно понять, что от него не получиться избавиться за счет мономов, содержащих $w_2$. Значит, тип будет постоянным для любого комплексного $z^0$ в том и только в том случае, если: $2a\conj z^0 + 3b z^0 \in\mathbb{R}, \forall z^0\in\mathbb{C}$ и $c =0 $. Или эквивалентно:
	
	\[\begin{cases}
		c=0,\\
		- 2 a + 3 b = 2 a - 3 b,
	\end{cases}\]
	откуда $(a:b: c) = (3:2:0)$.
	\smallskip
	
	Заметим, что поверхность $\mathcal{C}$ полиномиально-треугольным преобразованием
	\[\begin{cases}
		z = x_1 + i y_1 \to 2 z,\\
		w_2 = x_2 + i y_2\to 2 w_2 + 2 z^2,\\
		w_3 = x_3 + i y_3\to 3 w_3 + 2 z^3,\\
		w_4 = x_4 + i y_4\to 2 w_4 + 2 z^4,
	\end{cases}\]
	приводится к виду $\mathcal Q_{3,2,0}$.
\end{proof}
В дальнейшем потребуется знание того, как устроена $\mathrm{Aut}\, \mathcal C$, поэтому сформулируем
\begin{proposition}
	Группа голоморфных автоморфизмов $\mathcal{C}$ 6-мерна и параметризуется 4 вещественными параметрами и 1 комплексным:
	\[\begin{cases}
		z \to \lambda z + p,\\
		w_2 \to \lambda^2 w_2 + 2 i \lambda z \conj p + i |p|^2 + q_2,\\
		w_3 \to \lambda^3 w_3 + 4 \lambda^2 w_2 \Re p + 2 i \lambda ^2 z^2 \conj p + 2 i \lambda a (\conj p ^2 + 2|p|^2) + 2i p^2\conj p + q_3,\\
		w_4 \to \lambda^4 w_4 + 12 \lambda^3 w_3 \Re p + 12 \lambda^2 w_2 \Re{(p^2 + |p|^2)} + 4i \lambda^3 z^3 \conj p + 6i \lambda^2 z^2 (\conj p ^2 + 2|p|^2) +\\ \qquad\qquad\qquad + 4i \lambda z (6\Re{p\conj p ^2 + \conj p ^3}) + 3i |p|^4 + 4i \Re p^3\conj p + q_4,\\
		\lambda \in \mathbb{R}\setminus\{0\}, q_2, q_3, q_4 \in \mathbb{R}, p\in \mathbb{C}.
	\end{cases}\]
\end{proposition}
Обозначим $\mathcal{G}$ связную компоненту этой группы, содержащую единицу.

\section{Описание орбит действия $\mathcal G$}

\begin{definition}
	Гладкое подмногообразие $M\subset \mathbb R ^N$ назовем полуалгебраическим, если оно задается системой:
	\[M = \{x\in \mathbb R^N \mid P(x) =0,\; Q(x) > 0\},\]
	где $P(x), Q(x)$ -- полиномы.
\end{definition}

\begin{proposition}
	Орбиты действия группы $\mathcal G$ являются вещественными гладкими многообразиями размерности 5, если $(A, B, C, D) \in \mathcal{C}$, и 6 иначе. Более того, ниже будет показано, что они являются вещественными полуалгебраическими подмногообразиями и все орбиты, за исключением $\mathcal C$, имееют CR-тип $(2, 2)$.
\end{proposition}
\begin{proof}
	Если $(A, B, C, D) \in \mathcal{C}$, то по построению ее орбитой будет $\mathcal{C}$. Явное вычисление ранга вещественной матрицы
	\[\left(
	\begin{matrix}
		x & y & 2 u_2 & 2 v_2 & 3 u_3 & 3 v_3 & 4 u_4 & 4 v_4\\
		0 & 0 & 1 & 0 & 0 & 0 & 0 & 0\\
		0 & 0 & 0 & 0 & 1 & 0 & 0 & 0\\
		0 & 0 & 0 & 0 & 0 & 0 & 1 & 0\\
		1 & 0 & -2y & 2x & 4u_2 - 4xy & s_1 & s_2  & s_3 \\
		0 & 1 & 2x & 2y & 2(x^2-y^2) & s_4 & s_5 & s_6\\
	\end{matrix}
	\right),
	\]
	где
	\[s_1 = 2(2 v_2 + x^2-y^2),\; s_2 = 4(3 u_3 - 3x^2 y + y^3),\; s_3 = 4(3 v_3 + x^3 - x y^2),\] \[s_4 = 4xy,\; s_5 = 4(x^3 - 3x y^2),\; s_6 = 4(3x^2 y -y^3),
	\]
	составленной из 6 векторов, пораждающих алгебру Ли группы Ли $\mathcal G$, показывает, что в случае $(A,B,C,D) \notin \mathcal C$ её ранг равен 6, а значит орбитой будет гладкое шестимерное многообразие.
\end{proof}

А теперь получим уравнения, задающие орбиты. Для этого запишем координаты образа точки $(A, B, C, D) \in \mathbb{C}^4$ под действием  группы $\mathcal{G}$:
\[\begin{cases}
	z = \lambda A + p,\\
	w_2 = \lambda^2 B + 2 i \lambda A \conj p + i |p|^2 + q_2,\\
	w_3 = \lambda^3 C + 4 B \lambda^2 \Re p + 2 i \lambda ^2 A^2 \conj p + 2 i \lambda A (\conj p ^2 + 2|p|^2) + 2i p^2\conj p + q_3,\\
	w_4 = \lambda^4 D + 12 C \lambda^3\Re p + 12 B\lambda^2\Re{(p^2 + |p|^2)} + 4i A^3\lambda^3\conj p + 6i A^2\lambda^2(\conj p ^2 + 2|p|^2) +\\ \qquad\qquad\qquad + 4iA\lambda (6\Re{p\conj p ^2 + \conj p ^3}) + 3i |p|^4 + 4i \Re p^3\conj p + q_4.
\end{cases}\]
Из первых двух уравнений получаем, что 
\[\Im w_2 - |z|^2 = \lambda^2 (\Im B - |A|^2).\]
Отсюда видно, что $P(z, w_2, w_3, w_4) = \Im w_2 - |z|^2$ -- относительный инвариант веса 2 относительно действия группы $\mathbb R_{>0}:$
\[z\to \lambda z,\quad w_2 \to \lambda^2 w_2,\quad w_3 \to \lambda^3 w_3,\quad w_4 \to \lambda^4 w_4,\quad\lambda>0\] В зависимости от знака $P$ возможны 3 ситуации.
Для случая $P(A, B, C, D)>0$, вдоль всей орбиты имеем: $P(z, w_, w_3, w_4) > 0$ и 
\[\lambda = \dfrac{\sqrt{P(z, w_, w_3, w_4)}}{\sqrt{P(A, B, C, D)}}>0,\] поскольку $\lambda\in \mathcal G$. Пусть
\begin{align*}
	Q(z, w_1, w_2, w_3) =&\; \Im w_3 - 4\Re z \Im w_2 + 2\Re z^2\conj z,\\
	R(z, w_1, w_2, w_3) =&\; \Im w_4 - (3|z|^4 + 4\Re z^3 \conj z) + 12\Re z (2\Im w_2 \Re z - \Im w_3),
\end{align*}
тогда оставшиеся уравнения задают орбиту:
\[\mathcal O_{\mu, \sigma} = \begin{cases}
	Q(z, w_2, w_3, w_4) = \mu P(z, w_2, w_3, w_4)^{\frac{3}{2}},\\
	R(z, w_2, w_3, w_4) = \sigma P(z, w_2, w_3, w_4)^2,\\
	P(z, w_2, w_3, w_4)>0,
\end{cases}\]
где 
\[\mu = \dfrac{Q(A, B, C, D)}{P(A, B, C, D)^{\frac{3}{2}}} \in \mathbb{R},\]
\[\sigma = \dfrac{R(A, B, C, D)}{P(A, B, C, D)^2} \in \mathbb{R}.\]
Отметим, что $Q, R$ -- относительные инварианты весов 3 и 4 соответственно. Поверхность $\mathcal O_{\mu, \sigma}$ голоморфно-однородна, поэтому вычисление Блум~"---Грэм типа достаточно провести в одной точке, например $(0, i, i\mu, i\sigma)$.

В этой точке эрмитовы формы имеют вид:
\[\begin{cases}
	-\dfrac{3}{2}\mu |z_1|^2 + i (z_1 \conj z_2 -  z_2\conj z_1) + \dfrac{3}{16}\mu |z_2|^2,\\
	\dfrac{9}{2}i\mu (z_1\conj z_2 - z_2\conj z_1) + \dfrac{1}{2}\sigma |z_2|^2.
\end{cases}
\]
При $\mu = \sigma = 0$ эти формы линейно зависимы и орбита не является вполне невырожденной. В этом случае модельная поверхность орбиты, которую я буду обозначать $E$, задается следующим образом:
\[
\begin{cases}
	\Im w_1 =  z_1\conj z_2 + z_2\conj z_1,\\
	\Im w_2 = 2\Re(z_1^2\conj z_2 + 2 z_1 z_2\conj z_1),\\
\end{cases}	
\]
и в точке $0$ имеет тип по Блуму~"---Грэму $((2, 1), (3, 1))$. В случае, если $\mu^2 + \sigma^2 \ne 0$ модельной поверхностью будет модельная квадрика типа $(2, 2)$. Она определяется парой линейно независимых эрмитовых форм и линейно эквивалентна одной из 3 квадрик\cite{LobodaHermitianForm}:
$$Q_{\pm 1} = \{\Im w_1 = |z_1|^2 \pm |z_2|^2,\, \Im w_2 = z_1\conj z_2 + z_2\conj z_1\},$$
$$Q_0 = \{\Im w_1 = |z_1|^2,\, \Im w_2 = z_1\conj z_2 + z_2\conj z_1\}.$$
\begin{itemize}
	\item Если $\mu = 0,\,\sigma \ne 0$, то квадрика эквивалентна $Q_0$.
	\item Если $\mu \ne 0,\, \sigma = 0$, то квадрика эквивалентна $Q_{-1}$.
	\item Если $\sigma \geqslant 27\mu^2/16>0$, то квадрика эквивалентна $Q_1$.
	\item Если $\mu\ne 0, \,\sigma^2 > 81\mu^2/2 - 24\sigma>0$, то квадрика эквивалентна $Q_1$.
	\item Если $\mu\ne0,\,\sigma^2 = 81\mu^2/2 - 24\sigma>0$, то квадрика эквивалентна $Q_0$.
	\item Если $\mu\ne 0,\, 81\mu^2/2 - 24\sigma>\sigma^2>0$, то квадрика эквивалентна $Q_{-1}$.
\end{itemize}
Пусть теперь мы имеем $P(A, B, C, D) < 0$, тогда $P(z, w_2, w_3, w_4) < 0,$
\[\lambda = \dfrac{\sqrt{-P(z, w_2, w_3, w_4)}}{\sqrt{-P(A, B, C, D)}}\] и 
\[\mathcal O ^{\nu,\sigma} := \begin{cases}
	
	Q(z, w_2, w_3, w_4) = \mu(-P(z, w_2, w_3, w_4))^{\frac{3}{2}},\\
	R(z, w_2, w_3, w_4) = \sigma P(z, w_2, w_3, w_4)^2,\\
	P(z, w_2, w_3, w_4) < 0,
\end{cases}
\]
где 
$$\nu = \frac{Q(A, B, C, D)}{(-P(A, B, C, D))^\frac{3}{2}}\in\mathbb{R}.$$
В качестве пробной точки выберем $(0, -i, i\nu, i\sigma)$, тогда пара эрмитовых форм примет вид:
\[\begin{cases}
	\dfrac{3}{2}\nu |z_1|^2 + i z_1 \conj z_2 - i z_2\conj z_1 + \dfrac{3}{16}\nu |z_2|^2,\\
	-\dfrac{9}{2}i\nu (z_1\conj z_2 - z_2\conj z_1) + \dfrac{1}{2}\sigma |z_2|^2.
\end{cases}\]
Здесь похожая ситуация: при $\nu = \sigma = 0$ -- формы линейно зависимы и модельной поверхностью является $E$.

Иначе имеем линейную независимость и если:
\begin{itemize}
	\item $\nu = 0,\,\sigma \ne 0$, то квадрика эквивалентна $Q_0$,
	\item $\nu \ne 0,\,\sigma =0 $, то квадрика эквивалентна $Q_1$,
	\item $\nu,\,\sigma\ne 0, \, \sigma\leqslant 27\nu^2/16$, то квадрика эквивалентна $Q_1$,
	\item $\nu\ne 0,\, \sigma^2>24\sigma - 81\nu^2/2>0 $, то квадрика эквивалентна $Q_1$,
	\item $\nu\ne 0,\, \sigma^2 = 24\sigma - 81\nu^2/2>0 $, то квадрика эквивалентна $Q_0$,
	\item $\nu\ne 0,\, 24\sigma - 81\nu^2/2>\sigma^2>0 $, то квадрика эквивалентна $Q_{-1}$.
\end{itemize}

Пусть теперь $P(A, B, C, D) =0$ и значит $P(z, w_2, w_3, w_4)=0$ вдоль орбиты. Из третьего уравнения:
\[\Im w_3 - 2\Re z^2 \conj z = \lambda^3(\Im C - 2\Re A^2\conj A).\]
Вновь возможны 3 случая: $\Im C > 2\Re A^2\conj A,\, \Im C < 2\Re A^2\conj A$ и $\Im C = 2\Re A^2\conj A$. При $\Im C> 2\Re A^2\conj A$ орбиту, обозначаемую $\mathcal O _\rho$, можно задать системой:
\[\begin{cases}
	P(z, w_2, w_3, w_4) = 0,\\
	Q(z, w_2, w_3, w_4) > 0,\\
	R(z, w_2, w_3, w_4) = \rho Q(z, w_2, w_3, w_4)^{\frac{4}{3}},
\end{cases}
\] где
\[\rho = \dfrac{\Im D - (3|A|^4 + 4\Re A^3\conj A) + 12 \Re A (2\Re A^2\conj A - \Im C)}{(\Im C - 2\Re A^2\conj A)^{4/3}}\in\mathbb R.\]

Выпишем соответствующие эрмитовы формы в точке $(0, 0, i, i\rho):$
\[\begin{cases}
	|z_1|^2,\\
	3i\conj z_2 z_1 - 3iz_2\conj z_1 + \dfrac{1}{9}\rho|z_2|^2.
\end{cases}\]
Cоответствующая квадрика, эквивалентная $Q_1$, будет модельной для орбиты.

Если же $\Im C - 2\Re A^2\conj A < 0$, то орбиту можно задать следующим образом:
\[\mathcal O ^\tau := 
\begin{cases}
	P(z, w_2, w_3, w_4) = 0,\\
	Q(z, w_2, w_3, w_4) < 0,\\
	R(z, w_2, w_3, w_4) = \tau  (-Q(z, w_2, w_3, w_4))^{\frac{4}{3}},
\end{cases}\]
где
\[\tau = \dfrac{\Im D - (3|A|^4 + 4\Re A^3\conj A) + 12 \Re A (2\Re A^2\conj A - \Im C)}{(2\Re A^2\conj A - \Im C)^{4/3}}\in\mathbb R.\]
Аналогичное вычисление в точке $(0, 0, -i, i\tau)$ показывает, что соответствующая \, квадрика будет вполне невырожденной модельной поверхностью.

Наконец, если $\Im C = 2\Re A^2\conj A,$ то $\Im w_3 = 2\Re z^2\conj z$ и исключая орбиту $\mathcal C$ возможны:
\[\mathcal O _{\pm} :=
\begin{cases}
	P(z, w_2, w_3, w_4) = Q(z, w_2, w_3, w_4) = 0,\\
	\pm R(z, w_2, w_3, w_4) > 0.
\end{cases}\]
$\mathcal O_\pm$ -- не являются вполне невырожденными многообразиями. Более того, они голоморфно вырождены, поскольку их алгебры автоморфизмов $\mathrm{aut}\, \mathcal O _{\pm}$ содержат векторные поля вида $2\Re f(z, w_2, w_3, w_4)\frac{\partial}{\partial w_4},$ а значит они бесконечномерны. Их модельные поверхности совпадают между собой и задаются следующей системой:
\[F=\begin{cases}
	\Im w_2 = |z|^2,\\
	\Im w_3 = 2\Re z^2\conj z,\\
	z, w_4 \in \mathbb C.
\end{cases}\]

В результате проделанных вычислений имеем
\begin{theorem}
	Все модельные поверхности орбит действия $\mathcal{G}$ классифицируются следующим образом:
	\begin{itemize}
		\item $\mathcal O_{\mu, \sigma}$:
		\begin{enumerate}
			\item при $\mu = 0,\,\sigma = 0$ модельная поверхность эквивалентна $E$				
			\item при $\mu = 0,\,\sigma \ne 0$, то модельная поверхность эквивалентна $Q_0$.
			\item при $\mu \ne 0,\, \sigma = 0$, то модельная поверхность эквивалентна $Q_{-1}$.
			\item при $\sigma \geqslant 27\mu^2/16>0$, то модельная поверхность эквивалентна $Q_1$.
			\item при $\mu\ne 0, \,\sigma^2 > 81 mu^2/2 - 24\sigma>0$, то модельная поверхность эквивалентна $Q_1$.
			\item при $\mu\ne0,\,\sigma^2 = 81 mu^2/2 - 24\sigma>0$, то модельная поверхность эквивалентна $Q_0$.
			\item при $\mu\ne 0,\, 81 mu^2/2 - 24\sigma>\sigma^2>0$, то модельная поверхность эквивалентна $Q_{-1}$.
		\end{enumerate}
		\item $\mathcal O^{\nu,\sigma}$:
		\begin{enumerate}
			\item $\nu = 0,\,\sigma = 0$ модельная поверхность эквивалентна $E$,
			\item $\nu = 0,\,\sigma \ne 0$, то модельная поверхность эквивалентна $Q_0$,
			\item $\nu \ne 0,\,\sigma =0 $, то модельная поверхность эквивалентна $Q_1$,
			\item $\nu,\,\sigma\ne 0, \, \sigma\leqslant 27 \nu^2/16$, то модельная поверхность эквивалентна $Q_1$,
			\item $\nu\ne 0,\, \sigma^2>24\sigma - 81 \nu^2/2>0 $, то модельная поверхность эквивалентна $Q_1$,
			\item $\nu\ne 0,\, \sigma^2 = 24\sigma - 81 \nu^2/2>0 $, то модельная поверхность эквивалентна $Q_0$,
			\item $\nu\ne 0,\, 24\sigma - 81 \nu^2/2>\sigma^2>0 $, то модельная поверхность эквивалентна $Q_{-1}$.
		\end{enumerate}
		\item $\mathcal O_{\rho}$ и $\mathcal O^{\rho}$: модельная поверхность эквивалентна $Q_1$,
		\item $\mathcal O_{\pm}:$ модельная поверхность эквивалентна $F$,
		\item $\mathcal C:$ является модельной поверхностью.
		
	\end{itemize}
\end{theorem}

Рассмотрим теперь вопрос о попарной биголоморфной эквивалентности орбит. Заметим, что преобразование: $z \to -z,\, w_2 \to w_2,\, w_3 \to -w_3,\, w_4 \to w_4$ отображает:
\begin{align*}
	\mathcal O_{\mu,\sigma} &\to \mathcal O_{-\mu,\sigma},\\
	\mathcal O^{\nu,\sigma} &\to \mathcal O^{-\nu,\sigma},\tag{1}\label{equiv}\\
	\mathcal O_{\rho} &\to \mathcal O^{\rho}.
\end{align*}

Покажем, что кроме этих пар орбит, никакие другие вполне невырожденные не являются биголоморфно эквивалентными. Для этого нам понадобится пара лемм, аналогичные лемам 3.1 и 3.2 из \cite{BeloshapkaKossovskyOrbits}. Сформулируем их применительно к нашей ситуации.
\begin{lemma}
	Всякая вполне невырожденная орбита обладает конечномерной полиномиальной алгеброй инфинитезимальных автоморфизмов.
\end{lemma}

\begin{lemma}
	Любой биголоморфизм между вполне невырожденными орбитами задается бирациональным отображением $\mathbb C ^4$.
\end{lemma}

Их доказательства почти дословно повторяют доказательства из приведенной работы. Также аналогично, можно получить
\begin{proposition}
	\begin{enumerate}
		\item Никакие пары вполне невырожденных орбит не являются биголоморфно эквивалентными, за исключением (\ref{equiv}).
		\item Группа автоморфизмов вполне невырожденных орбит совпадает с $\mathcal G$.
	\end{enumerate}
\end{proposition}

В теореме классификация проведена по отношению к модельным поверхностям. Встает вопрос о сферичности орбит, т.е. о биголоморфной эквивалентности орбиты своей модельной поверхности. Справедливо следующее
\begin{proposition}
	Ни одна из вполне невырожденных орбит не является сферической.
\end{proposition}
\begin{proof}
	Рассмотрим орбиту $\mathcal O _{\mu. \sigma}$, при $\mu,\sigma \ne 0$. Для удобства переобозначим переменные: $Z_1:= z,\, Z_2 := w_2,\, W_1 := w_3,\, W_2 := w_4$. Введем градуировку:
	\[[Z_1] = [\conj Z_1] = [Z_2] = [\conj Z_2] = 1,\, [W_1] = [W_2] = 2.\]
	Пусть в окрестности какой-нибудь точки $\xi \in \mathcal O_{\mu, \sigma}$ система имеет вид:
	\[\begin{cases}
		\Im W_1 = H_1(Z_1, Z_2, \conj Z_1, \conj Z_2) + \ldots,\\
		\Im W_2 = H_2(Z_1, Z_2, \conj Z_1, \conj Z_2) + \ldots,\\
	\end{cases}\]
	где $H_j$ -- эрмитовы формы, а многоточие означает члены больших весов. $\mathcal O_{\mu, \sigma}$ является сферической тогда и только тогда, когда существует биголоморфное отображение вида:
	\[
	\begin{cases}
		F_1(Z_1, Z_2, W_1, W_2) = Z_1 + \sum\limits_{n=2}^\infty f^{(1)}_n (Z_1, Z_2, W_1, W_2),\\
		F_2(Z_1, Z_2, W_1, W_2) = Z_2 + \sum\limits_{n=2}^\infty f^{(2)}_n (Z_1, Z_2, W_1, W_2),\tag{2}\label{map}\\
		G_1(Z_1, Z_2, W_1, W_2) = W_1 + \sum\limits_{n=3}^\infty g^{(1)}_n (Z_1, Z_2, W_1, W_2),\\
		G_2(Z_1, Z_2, W_1, W_2) = W_2 + \sum\limits_{n=3}^\infty g^{(2)}_n (Z_1, Z_2, W_1, W_2),\\
	\end{cases}
	\]
	где $f^{(j)}_n, g^{(j)}_n,\, j=1,2$ -- однородные компоненты веса $n$, для которого справедлива система:
	\[
	\begin{cases}
		\Im G_1(Z, W) = H_1(F_1(Z, W), F_2(Z, W), \overline{F_1(Z, W)}, \overline{F_2(Z, W)}),\tag{3}\label{system}\\
		\Im G_2(Z, W) = H_2(F_1(Z, W), F_2(Z, W), \overline{F_1(Z, W)}, \overline{F_2(Z, W)}),
	\end{cases}\]
	при $Z = (Z_1, Z_2),\, W = (W_1, W_2),$
	\begin{align*}
		W_1 = U_1 + i (4\Re Z_1 \Im Z_2 - 2\Re {Z_1^2\conj Z_1} + \mu (\Im Z_2 - |Z_1|^2)^{3/2}),\\
		W_2 = U_2 + i (3|Z_1|^4 + 4\Re Z_1^3\conj Z_1 + 24(\Im Z_2 - |Z_1|^2)(\Re Z_1)^2 + \\\qquad\qquad + 12\mu\Re Z_1 (\Im Z_2 - |Z_1|^2)^{3/2} + \sigma(\Im Z_2 - |Z_1|^2)^2).
	\end{align*}
	Разложим левую и правую части системы (\ref{system}) на компоненты $K_{(m, n)}$ бистепени $(m, n)$ по $Z, \conj Z$. Вычисления показывают, что для любого отображения вида (\ref{map}) такого, что $K_{(1,0)}=0$, компонента $K_{(2, 1)}\ne 0$. Поэтому $\mathcal O_{\mu, \sigma}$ -- не является сферической. Для остальных вполне невырожденных орбит рассуждение аналогично.
\end{proof}
Автор благодарит своего научного руководителя В.К. Белошапку за внимание к работе и ценные замечания.


\begin{thebibliography}{99}
		\bibitem{Cartan}
			E. Cartan,
			<<Sur la geometrie pseudo-conforme des hypersurfaces de deux variables complexes>>, I Ann. Math. Pure Appl. (II Ann. Scuola Norm. Sup. Pisa) 11 (2), no. 4 (1), 19-90 (333-354),
		\bibitem{BeloshapkaKossovskiyClassification}
			V.K. Beloshapka, I.G. Kossovskiy,
			<<Classification of homogeneous CR-manifolds in dimension 4>>, Journal of Mathematical Analysis and Application, 374 (2011), 655-672.
		\bibitem{LobodaHomogeneous}
			А.В. Лобода, 
			<<Голоморфно однородные вещественные гиперповерхности в $\mathbb C^3$>>, Труды ММО, 81:2 (2020),  205–280.
		\bibitem{FiniteBlGr}
			В. К. Белошапка,
			<<$CR$~"---многообразия конечного Блум-Грэм-типа: метод модельной поверхности>>, Russian Journal of Mathematical Physics, Vol. 27, pp 155-174 (2020).
		\bibitem{UniversalModel}
		В. К. Белошапка,
		<<Универсальная модель вещественного подмногообразия>>, Математические заметки, 2004, том 75, выпуск 4, 507-522.
		\bibitem{BloomGraham}
			Th.\,Bloom, I.\,Graham, <<On Type Conditions for Generic Real Submanifolds of $\mathbb{C}^n$>>, Invent. Math. 40, 217-243 (1977).
		\bibitem{LobodaHermitianForm}
			А.В. Лобода,
			<<Порождающие вещественно-аналитические многообразия коразмерности 2 в $\mathbb C ^4$ и их биголоморфные отображения>>, Известия АН СССР, серия математическая, 1988, том 52, выпуск 5, страницы 970-990.
		\bibitem{BeloshapkaKossovskyOrbits}
			V.K. Beloshapka, I.G. Kossovsky,
			<<Homogeneous hypersurfaces in $\mathbb C^3$ associated with a model CR-cubic>>, Journal of Geometric Analysis, Vol. 20, pp 538-564 (2010).
	\end{thebibliography}
\end{document}